\documentclass{amsart}

\usepackage{amssymb,latexsym,amsmath,extarrows}
\usepackage{graphicx}
\usepackage{xcolor}
\usepackage[symbol]{footmisc}
\usepackage{float}

\newtheorem*{acknowledgement}{Acknowledgements}
\newtheorem{theorem}{Theorem}[section]
\newtheorem{lemma}[theorem]{Lemma}
\newtheorem{proposition}[theorem]{Proposition}
\newtheorem{remark}[theorem]{Remark}

\newtheorem{definition}[theorem]{Definition}
\newtheorem{corollary}[theorem]{Corollary}

\newtheorem{conjecture}[theorem]{Conjecture}
\newtheorem*{intu}{Intuition}

\newcommand{\al}{\alpha}
\newcommand{\be}{\beta}

\newcommand{\e}{\varepsilon}

\newcommand{\la}{\lambda}

\newcommand{\eps}{\epsilon}

\newcommand{\cs}{\mathcal S}

\newcommand{\cb}{\mathcal B}
\newcommand{\cq}{\mathcal Q}

\newcommand{\wh}{\widehat}

\newcommand{\ZR}{\mathbb{R}}
\newcommand{\ZZ}{\mathbb{Z}}

\newcommand{\ZB}{\mathbb{B}}
\newcommand{\ZS}{\mathbb{S}}

\newcommand{\eit}{e^{i t \Delta}}

\newcommand{\hichi}{\raisebox{0.7ex}{\(\chi\)}}

\begin{document}

\title[$L^2$ estimate of Schr\"odinger maximal function]{Sharp $L^2$ estimate of Schr\"odinger maximal function in higher dimensions}

\author{Xiumin Du}
\address{
University of Maryland\\
College Park, MD}
\email{xdu@math.umd.edu}

\author[R. Zhang]{Ruixiang Zhang}
\address{
University of Wisconsin-Madison\\
Madison, WI}
\email{ruixiang@math.wisc.edu}

\begin{abstract}
We show that, for $n\geq 3$, $\lim_{t \to 0} e^{it\Delta}f(x) = f(x)$ holds almost everywhere for all $f \in H^s (\mathbb{R}^n)$ provided that $s>\frac{n}{2(n+1)}$. Due to a counterexample by Bourgain, up to the endpoint, this result is sharp and fully resolves a problem raised by Carleson. Our main theorem is a fractal $L^2$ restriction estimate,
which also gives improved results on the size of divergence set of Schr\"odinger solutions, the Falconer distance set problem and the spherical average Fourier decay rates of fractal measures. The key ingredients of the proof include multilinear Kakeya estimates, decoupling and induction on scales.
\end{abstract}

\dedicatory{Dedicated to the memory of Jean Bourgain}

\maketitle

\section{Introduction}
The solution to the free Schr\"{o}dinger equation
\begin{equation}
  \begin{cases}
    iu_t - \Delta u = 0, &(x,t)\in \mathbb{R}^n \times \mathbb{R} \\
    u(x,0)=f(x), & x \in \mathbb{R}^n
  \end{cases}
\end{equation}
is given by
$$
  e^{it\Delta}f(x)=(2\pi)^{-n}\int e^{i\left(x\cdot\xi+t|\xi|^2\right)}\widehat{f}(\xi) \, d\xi.
$$

In \cite{lC}, Carleson proposed the problem of identifying the optimal $s$ for which $\lim_{t \to 0}e^{it\Delta}f(x)=f(x)$ almost everywhere whenever
$f\in H^s(\mathbb{R}^n),$ and proved convergence
for $s \geq \frac 1 4$ when $n=1$. Dahlberg and Kenig \cite{DK} then showed that this result is sharp. The higher dimensional case has since been studied by several authors \cite{aC,mC,pS,lV,jB,MVV,TV,sL,jB12,LR17,DG,jB16,LR17',DGL,DGLZ}. In particular, almost everywhere convergence holds if $s>\frac 12-\frac{1}{4n}$ when $n\geq 2$ ($n=2$ due to Lee \cite{sL} and $n\geq 2$ due to Bourgain \cite{jB12}). Recently Bourgain \cite{jB16} gave counterexamples showing that convergence can fail if $s<\frac{n}{2(n+1)}$. Since then, Guth, Li and the first author \cite{DGL} improved the sufficient condition when $n=2$ to the almost sharp range $s>\frac{1}{3}$. In higher dimensions ($n\geq 3$), Guth, Li and the authors \cite{DGLZ} proved the convergence for $s>\frac{n+1}{2(n+2)}$. 

In this article, we establish the following theorem, which is sharp up to the endpoint.

\begin{theorem}\label{thm-pc}
Let $n\geq 3$. For every  $f\in H^s(\mathbb R^n)$ with $s>\frac{n}{2(n+1)}$, $\lim_{t \to 0}e^{it\Delta}f(x)=f(x)$ almost everywhere.
\end{theorem}

We use $B^m(x, r)$ to represent a ball centered at $x$ with radius $r$ in $\ZR^m$. By a standard smooth approximation argument, Theorem \ref{thm-pc} is a consequence of the following estimate of the Schr\"odinger maximal function:

\begin{theorem} \label{thm-L2-0} 
Let $n\geq 3$. For any $s > \frac{n}{2(n+1)}$, the following bound holds: for any function $f \in H^s(\mathbb{R}^n)$,
\begin{equation} \label{eq-L2-0}
\left\| \sup_{0 < t \le 1} | e^{it \Delta} f| \right\|_{L^2(B^n(0,1))}  \le C_s \| f \|_{H^s(\mathbb{R}^n)}.
\end{equation}
\end{theorem}

Via a localization argument, Littlewood-Paley decomposition and parabolic rescaling, Theorem \ref{thm-L2-0} is reduced to the following theorem which we will prove in this paper:

\begin{theorem}\label{thm-L2}
Let $n\geq 3$. For any $\e >0$, there exists a constant $C_\e$ such that
\begin{equation}\label{eq-L2}
 \left\|\underset{0<t\leq R}\sup|e^{it\Delta}f|\right\|_{L^2(B^n(0,R))} \leq
C_\e R^{\frac{n}{2(n+1)}+\e} \|f\|_2
\end{equation}
 holds for all $R\geq 1$ and all $f$ with ${\rm supp}\widehat{f}\subset A(1)=\{\xi\in \ZR^n:|\xi|\sim 1\}$.
\end{theorem}

\begin{remark}
When $n=1,2$, our proof of Theorem \ref{thm-L2} remains valid and recovers the almost sharp results of the pointwise convergence problem. However, the sharp $L^2$ estimates of the Schr\"odinger maximal function are not as strong as the previous sharp $L^p$ estimates in the cases $n=1,2$:
\begin{equation}\label{L4}
\left\| \sup_{t>0} | e^{it \Delta} f| \right\|_{L^4(\ZR)}  \le C \| f \|_{H^{1/4}(\mathbb{R})}\,, \quad \text{\cite[Kenig-Ponce-Vega]{KPV}}\,,
\end{equation}
and
\begin{equation}\label{L3}
\left\| \sup_{0 < t \le 1} | e^{it \Delta} f| \right\|_{L^3(\ZR^2)}  \le C_s \| f \|_{H^{s}(\mathbb{R}^2)}\,, \forall s>\frac 13, \quad \text{\cite[D.-Guth-Li]{DGL}} \footnote[1]{
The global $L^3$ estimate \eqref{L3} follows easily from the local $L^3$ estimate in \cite{DGL}, via a localization argument using wave packet decomposition.} \,.
\end{equation}
Testing with the standard examples used in restriction theory seems to
suggest that the following estimate holds for all $n\geq 1$:
\begin{equation}\label{Lp}
\left\| \sup_{0 < t \le 1} | e^{it \Delta} f| \right\|_{L^{\frac{2(n+1)}{n}}(\ZR^n)}  \le C \| f \|_{H^{\frac{n}{2(n+1)}}(\mathbb{R}^n)}\,.
\end{equation}

From \eqref{L4} and \eqref{L3} we see that \eqref{Lp} is true for $n=1$, and is true up to the endpoint for $n=2$. However, the estimate \eqref{Lp} fails in higher dimensions. In a recent work of Kim, Wang and the authors \cite{DKWZ}, by looking at Bourgain's counterexample \cite{jB16} in every intermediate dimension, we showed that the following local estimate
\begin{equation}\label{Lp-loc}
\left\| \sup_{0 < t \le 1} | e^{it \Delta} f| \right\|_{L^p(B^n(0,1))}  \le C_s \| f \|_{H^{s}(\mathbb{R}^n)}\,, \forall s>\frac {n}{2(n+1)}
\end{equation}
fails if $p> p_0:= 2+\frac{4}{(n-1)(n+2)}$.
Note that $\frac{2(n+1)}{n} > p_0$ when $n\geq 3$ and henceforth \eqref{Lp} fails. To our best knowledge, the following two problems are still open when $n\geq 3$: determine the optimal $p=p(n)$ for which we can have (\ref{Lp-loc}) and identify the optimal $s=s(n,p)$ for which \eqref{Lp-loc} with $p>2$ fixed holds.
\end{remark}

\begin{remark}
In our proof of \eqref{eq-L2}, no typical $L^2$ arguments such as Plancherel and $TT^*$ are invoked to take advantage of the particular use of the $L^2$ norm on the left hand side of \eqref{eq-L2}. In fact, the $L^2$ norm  will be converted to $L^p$ norm (see Proposition \ref{thm-main}), where $p=\frac{2(n+1)}{n-1}$ is the sharp exponent for the $l^2$ decoupling theorem in dimension $n$. The $L^2$ is used on the left hand side of \eqref{eq-L2} mostly because the numerology adds up favorably for that space.
\end{remark}

By lattice $L$-cubes we mean cubes of the form $l+[0,L]^n$ with $l\in (L\ZZ)^n$. Our main result is the following fractal $L^2$ restriction estimate, from which Theorem \ref{thm-L2} follows. 

\begin{theorem} \label{thm-L2X}
Let $n\geq 1$.
For any $\e>0$, there exists a constant\ $C_\e$ such that the following holds for all $R\geq 1$ and all $f$ with ${\rm supp}\widehat{f}\subset B^n(0,1)$. 
Suppose that $X=\bigcup_k B_k$ is a union of lattice unit cubes in $B^{n+1}(0,R)$ and each lattice $R^{1/2}$-cube intersecting $X$ contains $\sim \lambda$ many unit cubes in $X$. Let $1\leq \alpha\leq n+1$ and $\gamma$ be given by
\begin{equation}\label{ga-L2}
\gamma:=\max_{ \underset{x'\in \ZR^{n+1},r\geq 1}{B^{n+1}(x',r)\subset B^{n+1}(0,R)}} \frac{\#\{B_k : B_k \subset B(x',r)\}}{r^\alpha}\,.
\end{equation}
Then
\begin{equation} \label{eq-L2X}
\|\eit f\|_{L^2(X)} \leq C_\e \gamma^{\frac{2}{(n+1)(n+2)}} \lambda^{\frac{n}{(n+1)(n+2)}} R^{\frac{\alpha}{(n+1)(n+2)}+\e}\|f\|_2\,.
\end{equation}
\end{theorem}

Note that in Theorem \ref{thm-L2X}, $\la\leq \gamma R^{\alpha/2}$.
As a direct result of Theorem \ref{thm-L2X},
there holds a slightly weaker fractal $L^2$ restriction estimate. It has a relatively simpler statement:

\begin{corollary}
\label{cor-L2X}
Let $n\geq 1$. For any $\e>0$, there exists a constant $C_\e$ such that the following holds for all $R\geq 1$ and all $f$ with ${\rm supp}\widehat{f}\subset B^n(0,1)$. Suppose that $X=\bigcup_{k} B_k$
is a union of lattice unit cubes in $B^{n+1}(0,R)$.
Let $1\leq \al \leq n+1$ and $\gamma$ be given by
\begin{equation}\label{ga-L2'}
\gamma:=\max_{ \underset{x'\in \ZR^{n+1},r\geq 1}{B^{n+1}(x',r)\subset B^{n+1}(0,R)}} \frac{\#\{B_k : B_k \subset B(x',r)\}}{r^\alpha}\,.
\end{equation}
Then
\begin{equation} \label{eq-L2X'}
\|\eit f\|_{L^2(X)} \leq C_\e \gamma^{\frac{1}{n+1}} R^{\frac{\alpha}{2(n+1)}+\e}\|f\|_2\,.
\end{equation}
\end{corollary}

We will see that Corollary \ref{cor-L2X} is sufficient to derive the sharp $L^2$ estimate of Schr\"odinger maximal function (Theorem \ref{thm-L2}) and all other applications in Section \ref{sec-app}. This corollary can also be proved directly by a slightly simpler argument. The case $n=1$ of Corollary \ref{cor-L2X} can be recovered using the ingredients in Wolff's paper \cite{W}. 
See Subsection \ref{sec-rmk} for a discussion.

Nevertheless, Theorem \ref{thm-L2X} has two advantages compared to Corollary \ref{cor-L2X}. Firstly, it gives us a better $L^2$ restriction estimate if the set $X$ of unit cubes is fairly sparse at the scale $R^{1/2}$. Secondly, it tells us some geometric information about a set $X$ of unit cubes when $\|\eit f\|_{L^2(X)}$ is comparable to $\|\eit f\|_{L^2(B(0,R))}$. For example, taking $\alpha = n+1$ (hence $\gamma \lesssim 1$) we have:

\begin{corollary}\label{cor-fillin}
Let $n\geq 1$. Suppose that $X=\bigcup_{k} B_k$ is a union of lattice unit cubes in $B^{n+1}(0,R)$ and each lattice $R^{1/2}$-cube intersecting $X$ contains $\sim \lambda$ many unit cubes in $X$. Suppose there is a function $f$ with ${\rm supp}\wh f \subset B^n(0,1)$ and $\|f\|_2 \neq 0$ such that
$\|\eit f\|_{L^2(X)} \gtrsim R^{1/2}\|f\|_2$. Then $\lambda \gtrapprox R^{\frac{n+1}{2}}$.
\end{corollary}

As a remark, the scale $R^{1/2}$ in Corollary \ref{cor-fillin} is the largest one can have. Indeed, with the assumption of the corollary, the unit cubes in $X$ do not have to almost fill $R^{\beta}$-cubes completely for $\beta > 1/2$. One can see this from the Knapp example where we only have one wave packet.

To prove our main result - Theorem \ref{thm-L2X}, we will use a broad-narrow analysis, which has similar spirit as the techniques in the work of Bourgain-Guth \cite{BG}, Bourgain \cite{jB12}, Bourgain-Demeter \cite{BD} and Guth \cite{G2}. 

In the broad case, we can exploit the transversality and apply the multilinear refined Strichartz estimate, which is a result obtained by Guth, Li and the authors in \cite{DGLZ} (see \cite{DGL, DGOWWZ, DGLZ} for applications of the refined Strichartz estimate). In the narrow case, we use the $l^2$ decoupling theorem of Bourgain-Demeter \cite{BD} in a lower dimension and perform induction on scales. The way we do induction has its roots in the proof of the linear refined Strichartz estimate, due to Guth, Li and the first author (essentially proved in \cite{DGL}, see \cite{DGLZ} for the statement in the general setting).

Our method is related to Bourgain's in \cite{jB12}, where he has a similar broad-narrow analysis, (Here we have the size of the small ball being $K^2$ instead of $K$ as in \cite{jB12} for a technical issue similar to what one has in \cite{BD,G2}). He applied multilinear restriction to control the broad part in the sharp range $s>\frac{n}{2(n+1)}$ (except the endpoint). He speculated from this that the above range of $s$ might be sharp (see the end of the introduction in \cite{jB16}). In \cite{jB12} the narrow part was handled following the general approach from \cite{BG}, which gives non-sharp estimates. Historically, one could view the present non-endpoint solution to Carleson's problem as building on \cite{jB12}, providing a subtler way of handling the narrow part and proving Corollary \ref{cor-L2X}. For the stronger Theorem \ref{thm-L2X} and Corollary \ref{cor-fillin}, one needs a different ingredient, namely the multilinear refined Strichartz in \cite{DGLZ}, to handle the broad part.

In Section \ref{sec-app} we show how Corollary \ref{cor-L2X} and Theorem \ref{thm-L2} follow from Theorem \ref{thm-L2X},
and we also present applications of Theorem \ref{thm-L2X} to other problems - bounding the size of the divergence set of Schr\"odinger solutions (Theorem \ref{thm-pc-f}), the Falconer distance set problem (Theorem \ref{thm-falc} and \ref{thm-falc1}) and the spherical average Fourier decay rates of fractal measures (Theorem \ref{thm-avgdec}).
We prove Theorem \ref{thm-L2X} in Section 
\ref{sec-pf}.

\vspace{.1in}

\noindent \textbf{Notation.} We write $A\lesssim B$ if $A\leq CB$ for some absolute constant $C$, $A \sim B$ if $A\lesssim B$ and $B\lesssim A$; $A\ll B$ if $A$ is much less than $B$; $A\lessapprox B$ if $A\leq C_\e R^\e B$ for any $\e>0, R>1$. Sometimes we also write $A\lesssim B$ if $A\leq C_\e B$ for some constant $C_\e$ depending on $\e$ (when the dependence on $\e$ is unimportant).

By an $r$-ball (cube) we mean a ball (cube) of radius (side length) $r$. An $r\times\cdots\times r \times L$-tube (box) means a tube (box) with radius (short sides length) $r$ and length $L$. For a set $\cs$, $\#\cs$ denotes its cardinality.

\begin{acknowledgement}
The authors would like to thank Larry Guth and Xiaochun Li for several discussions. They also thank Larry Guth for making some historical remarks, as well as sharing his lecture notes on decoupling online, from which they got much inspiration. The second author would like to thank Jean Bourgain and Zihua Guo who introduced the problem to him. The authors are also indebted to Daniel Eceizabarrena  and  Luis Vega for a discussion on the history of the Schr\"odinger maximal estimate in dimension $1+1$.

The material is based upon work supported by the National Science Foundation under Grant No. 1638352, the Shiing-Shen Chern Fund and the James D. Wolfensohn Fund while the authors were in residence at the Institute for Advanced Study during the academic year 2017-2018.
\end{acknowledgement}

\section{Applications of Theorem \ref{thm-L2X}} \label{sec-app} 
\setcounter{equation}0

\subsection{Sharp $L^2$ estimate of Schr\"odinger maximal function} In this subsection, we show how Corollary \ref{cor-L2X} and Theorem \ref{thm-L2} follow from Theorem \ref{thm-L2X}, via the dyadic pigeonholing argument and the locally constant property.

\begin{proof}[Proof of (Theorem \ref{thm-L2X} $\implies$ Corollary \ref{cor-L2X})]
Given $X=\bigcup_{k} B_k$, a union of lattice unit cubes in $B^{n+1}(0,R)$ satisfying the assumptions of Corollary \ref{cor-L2X}, we sort the lattice $R^{1/2}$-cubes in $\ZR^{n+1}$ intersecting $X$ by the number $\lambda$ of unit cubes $B_k$ contained in it. Since $1\leq \lambda \leq R^{O(1)}$, there are only $O(\log R)$ choices for the dyadic number $\lambda$. So we can choose a dyadic number $\lambda$ and a subset $\cb_\lambda$ of $\{B_k\}$ such that for each unit cube $B$ in $\cb_\lambda$, the lattice $R^{1/2}$-cube containing it contains $\sim \lambda$ many unit cubes from $\cb_\lambda$ and
$$
\|\eit f\|_{L^2(X)} \lessapprox \|\eit f\|_{L^2(\bigcup_{B\in \cb_{\lambda}} B)}\,.
$$
By applying Theorem \ref{thm-L2X} to $\|\eit f\|_{L^2(\bigcup_{B\in \cb_{\lambda}} B)}$, we get
\begin{equation*}
\|\eit f\|_{L^2(X)} \lessapprox \gamma^{\frac{2}{(n+1)(n+2)}} \lambda^{\frac{n}{(n+1)(n+2)}} R^{\frac{\alpha}{(n+1)(n+2)}}\|f\|_2\,,
\end{equation*}
and \eqref{eq-L2X'} follows from the fact that $\lambda \leq \gamma R^{\alpha/2}$.
\end{proof}

\begin{proof}[Proof of (the case $\alpha=n$ of Corollary \ref{cor-L2X} $\implies$ Theorem \ref{thm-L2})] 
We will show that
\begin{equation} \label{linL2}
\left\|\sup_{0<t<R} |\eit f|\right\|_{L^2(B^n(0,R))} \lessapprox R^{\frac{n}{2(n+1)}} \|f\|_2
\end{equation}
holds for all $R\geq 1$ and all $f$ with Fourier support in $B^n(0,1)$.

By viewing $|\eit f(x)|$ essentially as constant on unit balls\footnote[2]{We refer the readers to \cite[Sections 2-5]{BG} for the standard formalism of this locally constant property.},  we can find a set $X$ described as follows: $X$ is a union of unit balls in $B^n(0,R)\times [0,R]$ satisfying the property that each vertical thin tube of dimensions $1\times \cdots \times 1 \times R$ contains exactly one unit ball in $X$, and 
\begin{equation}
\left\| \sup_{0<t< R} |\eit f| \right\|_{L^2(B^n(0,R))} \lessapprox \|\eit f\|_{L^2(X)}\,.
\end{equation}
The desired estimate \eqref{linL2} follows by applying Corollary \ref{cor-L2X} to $\|\eit f\|_{L^2(X)}$ with $\al=n$ and $\gamma \lesssim 1$.
\end{proof}

\subsection{Other applications} \label{sec-other}
By formalizing the locally constant property, from Corollary \ref{cor-L2X} we derive some weighted $L^2$ estimates - Theorem \ref{thm-L2-al} and \ref{thm-L2-al'}, which in turn have applications to several problems described below.

\begin{definition}
Let $\al\in(0,d]$. We say that $\mu$ is an \emph{$\al$-dimensional measure} in $\ZR^d$ if it is a probability measure supported in the unit ball $B^d(0,1)$ and satisfies that
\begin{equation}
\mu(B(x,r))\leq C_\mu r^\al, \quad \forall r>0, \quad \forall x\in \ZR^d.
\end{equation}
\end{definition}

Denote $d\mu_R(\cdot):=R^\al d\mu(\frac \cdot R)$.

\begin{theorem} \label{thm-L2-al}
Let $n\geq 1,\al\in(0,n]$ and $\mu$ be an $\al$-dimensional measure in $\ZR^n$. Then
\begin{equation} \label{L2-al}
\left\|\sup_{0<t<R} |\eit f|\right\|_{L^2\left(B^n(0,R);d\mu_R(x)\right)} \lessapprox R^{\frac{\al}{2(n+1)}} \|f\|_2\,,
\end{equation}
whenever $R\geq 1$ and $f$ has Fourier support in $B^n (0, 1)$.
\end{theorem}

\begin{theorem} \label{thm-L2-al'}
Let $n\geq 1,\al\in(0,n+1]$ and $\mu$ be an $\al$-dimensional measure in $\ZR^{n+1}$. Then
\begin{equation} \label{L2-al'}
\left\|\eit f\right\|_{L^2\left(B^{n+1}(0,R);d\mu_R(x,t)\right)} \lessapprox R^{\frac{\al}{2(n+1)}} \|f\|_2\,,
\end{equation}
whenever $R\geq 1$ and $f$ has Fourier support in $B^{n} (0, 1)$.
\end{theorem}

We defer the proof of these weighted $L^2$ estimates to the end of this subsection. Let's first see their applications. We omit history and various previous results on the following three problems and refer the readers to \cite{DGLZ,DGOWWZ,LR} and the references therein. 

\noindent \textbf{(I) Hausdorff dimension of divergence set of Schr\"odinger solutions}

A natural refinement of Carleson's problem was initiated by Sj\"ogren
and Sj\"olin \cite{SS}: determine the size of divergence set, in particular, consider
$$
\al_n(s):= \sup_{f\in H^s(\ZR^n)} {\rm dim} \left\{x\in \ZR^n: \lim_{t \to 0}e^{it\Delta}f(x)\neq f(x) \right\}\,,
$$
where ${\rm dim}$ stands for the Hausdorff dimension.

The following theorem is a direct result of Theorem \ref{thm-L2-al} (c.f. \cite{DGLZ, LR}). When $n=2$, it recovers the corresponding result derived from the sharp $L^3$ estimate of the Schr\"odinger maximal function in D.-Guth-Li \cite{DGL}. When $n\geq 3$, it improves the previous best known result in D.-Guth-Li-Z. \cite{DGLZ}.

\begin{theorem} \label{thm-pc-f}
Let $n\geq 2$. Then
\begin{equation}
\al_n(s)\leq n+1-\frac{2(n+1)s}{n}, \quad
\frac{n}{2(n+1)}<s<\frac n 4\,.
\end{equation}
\end{theorem}

\noindent \textbf{(II) Falconer distance set problem}

Let $E\subset\mathbb{R}^d$ be a compact subset, its distance set $\Delta(E)$ is defined by 
$$
\Delta(E):=\{|x-y|:x,y\in E\}\,.
$$

\begin{conjecture} [Falconer \cite{F}]
Let $d\geq 2$ and $E\subset\mathbb{R}^d$ be a compact set. Then 
$$
{\rm dim}(E)> \frac d 2 \Rightarrow |\Delta(E)|>0.
$$
Here $|\cdot|$ denotes the Lebesgue measure and ${\rm dim}(\cdot)$ is the Hausdorff dimension.
\end{conjecture}

Following a scheme due to Mattila (c.f. \cite[Proposition 2.3]{DGOWWZ}), Theorem \ref{thm-L2-al'} implies the following result towards Falconer's conjecture. When $d=2,3$, this recovers the previous best known results of Wolff (d=2, \cite{W}) and  D.-Guth-Ou-Wang-Wilson-Z. (d=3, \cite{DGOWWZ}), via a different approach. In the case $d\geq 4$, this improves the previous best known result in \cite{DGOWWZ}:

\begin{theorem} \label{thm-falc}
Let $d\geq 2$ and $E\subset\mathbb{R}^d$ be a compact set with 
$$
{\rm dim}(E)> \frac{d^2}{2d-1}=\frac d 2 +\frac 1 4 + \frac{1}{8d-4}\,.
$$
Then $|\Delta(E)|>0$.
\end{theorem}

By applying a very recent work of Liu \cite[Theorem 1.4]{Liu}, Theorem \ref{thm-L2-al'} also implies the following result for the pinned distance set problem, with the same threshold:

\begin{theorem} \label{thm-falc1}
Let $d\geq 2$ and $E\subset \mathbb{R}^d$ be a compact set with
$$
{\rm dim}(E)> \frac{d^2}{2d-1}=\frac d 2 +\frac 1 4 + \frac{1}{8d-4}\,.
$$
Then there exists $x\in E$ such that its pinned distance set
\[
\Delta_x(E):=\{|x-y|:\,y\in E\}
\]has positive Lebesgue measure.
\end{theorem}

\noindent \textbf{(III) Spherical average Fourier decay rates of fractal measures}

Let $\be_d(\al)$ denote the supremum of the numbers $\be$ for which 
\begin{equation} \label{eq:AvrDec}
\left\|\widehat \mu (R\cdot)\right\|_{L^2(\ZS^{d-1})}^2 \leq C_{\al,\mu} R^{-\be}
\end{equation}
whenever $R> 1$ and $\mu$ is an $\al$-dimensional measure in $\ZR^d$. The problem of identifying the precise value of $\beta_d(\al)$ was proposed by Mattila \cite{M04}. 

A lower bound of $\beta_d(\al)$ as in Theorem \ref{thm-avgdec} follows from Theorem \ref{thm-L2-al'} (c.f. \cite[Remark 2.5]{DGOWWZ}). When $d=2$, this recovers the sharp result of Wolff \cite{W}. When $d=3$ and $\al\in(\frac 32, 2]$, this recovers the previous best known result of D.-Guth-Ou-Wang-Wilson-Z. \cite{DGOWWZ}. In the case $d=3, \al\in(2,3)$ or $d\geq 4,\alpha\in(d/2,d)$, this improves the previous best known result in \cite{DGOWWZ}.

\begin{theorem} \label{thm-avgdec}
Let $d\geq 2$ and $\alpha\in (\frac d 2,d)$. Then 
$$
\beta_d(\al) \geq \frac{(d-1)\al}{d}\,.
$$
\end{theorem}

\vspace{.1in}

The proofs of Theorems \ref{thm-L2-al} and \ref{thm-L2-al'} are entirely similar and we only do the proof of the former here, which is slightly more involved.

\begin{proof}[Proof of Theorem \ref{thm-L2-al}]
Denote $\eit f(x)$ by $Ef(x,t)$, and $(x,t)$ by $\tilde x$.
Since $\mathrm{supp} \wh{f} \subseteq B^n(0, 1)$, we have $\mathrm{supp} \wh{Ef} \subseteq B^{n+1}(0, 1)$. Thus there exists a Schwartz bump function $\psi$ on $\ZR^{n+1}$ (we require $\wh{\psi} \equiv 1$ on $B^{n+1} (0, 100)$) such that $(Ef)^2 = (Ef)^2 * \psi$.

The function $\max_{|\tilde y- \tilde x|\leq e^{100n}} |\psi (\tilde y)|$ is rapidly decaying. We call it $\psi_1 (\tilde x)$. Note also that any $(x, t)$ in $\ZR^{n+1}$ belongs to a unique integral lattice cube whose center we denote by $\tilde m =(m,m_{n+1})= (m_1, \ldots, m_{n+1}) = \tilde m(x, t)$.

Then we have
\begin{equation}\label{eqn1provingL2al}
\begin{split}
&\left\|\sup_{0<t<R} |\eit f|\right\|_{L^2\left(B^n(0,R);d\mu_R\right)}^2\\
= & \int_{B^n(0,R)} \sup_{0 <t < R} |E f(x,t)|^2 d\mu_R(x)\\
\leq & \int_{B^n(0,R)} \sup_{0 <t < R} \left(|Ef|^2*|\psi|\right)(x,t) d\mu_R(x)\\
\leq & \int_{B^n(0,R)} \sup_{0 <t < R} \left(|E f|^2*\psi_1\right)(\tilde m(x, t)) d\mu_R(x)\\
\leq & \sum_{\underset{|m_i|\leq R}{m=(m_1, \ldots, m_n) \in \ZZ^n}, } \left(\int_{|x-m|\leq 10} d\mu_R(x)\right)
\cdot \sup_{\underset{0\leq m_{n+1}\leq R}{m_{n+1} \in \ZZ} } (|E f|^2*\psi_1)(m,m_{n+1}).
\end{split}
\end{equation}

For each $m \in \ZZ^n$, let $b(m)$ be an integer in $[0, R]$ such that
$$
\sup_{\underset{0\leq m_{n+1}\leq R}{m_{n+1} \in \ZZ} } (|E f|^2*\psi_1)(m,m_{n+1}) = (|E f|^2*\psi_1)(m,b(m)).
$$
Also we assume $\|f\|_2 = 1$ so $|\eit f|$ is uniformly bounded pointwisely. For each $m\in\ZZ^n$ we define $$\nu_m := \int_{|x-m|\leq 10} d\mu_R(x) \lesssim 1\,.$$ 
By \eqref{eqn1provingL2al}, we have
\begin{equation}\label{eqn2provingL2al}
\begin{split}
&\left\|\sup_{0<t<R} |\eit f|\right\|_{L^2\left(B^n(0,R);d\mu_R\right)}^2 \\
\lesssim & \sum_{\underset{\nu \in [R^{-100n}, 1]} {\nu \text{ dyadic}} } \sum_{ \underset{\nu_m \sim \nu}{m\in \ZZ^n, |m_i|\leq R}} \nu \cdot (|E f|^2*\psi_1)(m, b(m)) + R^{-90n}.
\end{split}
\end{equation}

For each dyadic $\nu$, denote $A_{\nu} = \{m \in \ZZ^n : |m_i|\leq R, \nu_m \sim \nu\}$. Performing a dyadic pigeonholing over $\nu$ we see that there exists a dyadic $\nu \in [R^{-100n}, 1]$ such that for any small $\varepsilon > 0$,
\begin{equation}\label{eqn3provingL2al}
\begin{split}
&\left\|\sup_{0<t<R} |\eit f|\right\|_{L^2\left(B^n(0,R);d\mu_R\right)}^2\\
\lessapprox & \sum_{m \in A_{\nu}} \nu\cdot (|E f|^2*\psi_1)(m, b(m)) + R^{-89n}\\
\lesssim & \sum_{m \in A_{\nu}} \nu \cdot\left(\int_{B^{n+1} ((m,b(m)), R^{\varepsilon})}|E f|^2\right) + R^{-89n}\\
\lesssim & \nu \cdot \int_{\bigcup_{m \in A_{\nu}} B^{n+1} ((m,b(m)), R^{\varepsilon})}|E f|^2 + R^{-89n}.
\end{split}
\end{equation}

Consider the set $X_{\nu} = \bigcup_{m \in A_{\nu}} B^{n+1} ((m,b(m)), R^{\varepsilon})$. It is a union of a collection of distinct $R^{\varepsilon}$-balls and at the same time, it is also a union of unit balls. These balls' projection onto the $(x_1, \ldots, x_n)$-plane are essentially disjoint (a point can be covered $\lesssim R^{\varepsilon}$ times). For every $r> R^{2\varepsilon}$ by the definition of $A_{\nu}$, the intersection of $X_{\nu}$ and any $r$-ball can be contained in no more than $R^{10n\varepsilon} \nu^{-1} r^{\alpha}$ disjoint $R^{\varepsilon}$-balls. Hence we can apply Corollary \ref{cor-L2X} to $X_{\nu}$ with $\gamma \lesssim R^{100 n \varepsilon} \nu^{-1}$ and $\alpha$. With (\ref{eqn3provingL2al}) this gives
\begin{equation}\label{eqn4provingL2al}
\left\|\sup_{0<t<R} |\eit f|\right\|_{L^2\left(B^n(0,R);d\mu_R\right)}^2 \lessapprox \nu^{\frac{n-1}{n+1}} R^{\frac{\al}{n+1}} \|f\|_2^2 \lesssim R^{\frac{\al}{n+1}} \|f\|_2^2.
\end{equation}
This concludes the proof.
\end{proof}

\section{Main inductive proposition and proof of Theorem \ref{thm-L2X}} \label{sec-pf}
\setcounter{equation}0

To prove Theorem \ref{thm-L2X}, we will use a broad-narrow analysis, which involves inductions. To make everything work we introduce another parameter $K$ and state the theorem in a slightly different way. We say that a collection of quantities are dyadically constant if all the quantities are in the same interval of the form $[2^j,2^{j+1}]$, where $j$ is an integer.
This is our main inductive proposition:

\begin{proposition} \label{thm-main}
Let $n\geq 1$. For any $0<\e<1/100$, there exist constants $C_\e$ and $0<\delta=\delta(\e) \ll \e$ (e.g. $\delta=\e^{100}$) such that the following holds for all $R\geq 1$ and all $f$ with ${\rm supp}\widehat{f}\subset B^n(0,1)$. Let $p=\frac{2(n+1)}{n-1}$ ($p=\infty$ when $n=1$). Suppose that $Y=\bigcup_{k=1}^M B_k$ is a union of lattice $K^2$-cubes in $B^{n+1}(0,R)$ and each lattice $R^{1/2}$-cube intersecting $Y$ contains $\sim \lambda$ many $K^2$-cubes in $Y$, where $K=R^\delta$. Suppose that
$$
\|\eit f\|_{L^p(B_k)} \text{ is dyadically a constant in } k=1,2,\cdots,M.
$$
Let $1\leq\alpha\leq n+1$ and $\gamma$ be given by
\begin{equation}\label{ga}
\gamma:=\max_{ \underset{x'\in \ZR^{n+1},r\geq K^2}{B^{n+1}(x',r)\subset B^{n+1}(0,R)}} \frac{\#\{B_k : B_k \subset B(x',r)\}}{r^\alpha}\,.
\end{equation}
Then
\begin{equation} \label{eq-main}
\|\eit f\|_{L^p(Y)} \leq C_\e M^{-\frac{1}{n+1}}\gamma^{\frac{2}{(n+1)(n+2)}} \lambda^{\frac{n}{(n+1)(n+2)}} R^{\frac{\alpha}{(n+1)(n+2)}+\e}\|f\|_2\,.
\end{equation}
\end{proposition}

Theorem \ref{thm-L2X} follows from Proposition \ref{thm-main} by a dyadic pigeonholing argument:

\begin{proof}[Proof of (Proposition \ref{thm-main} $\implies$ Theorem \ref{thm-L2X})] 
Given $X=\bigcup_{k} B_k$, a union of lattice unit cubes satisfying the assumptions of Theorem \ref{thm-L2X}, we sort these unit cubes $B_k$ according to the value of $\|\eit f\|_{L^p(B_k)}$. Assuming $\|f\|_2=1$, there are only $O(\log R)$ significant dyadic choices for this value. Therefore we can choose $X'\subset X$, a union of unit cubes $B$, such that 
$$
\left\{\|\eit f\|_{L^p(B)}: B\in X'\right\} \text{ are dyadically constant}
$$
and 
$$
\|\eit f\|_{L^2(X)} \lessapprox \|\eit f\|_{L^2(X')}\,.
$$
Let $M$ be the total number of unit cubes $B$ in $X'$.
Since $f$ has Fourier support in the unit ball, by locally constant property, $|\eit f|$ is essentially constant on unit balls. Therefore, the estimate \eqref{eq-L2X} is equivalent to
\begin{equation} \label{eq-LpX}
\|\eit f\|_{L^p(X')} \lessapprox  M^{-\frac{1}{n+1}}\gamma^{\frac{2}{(n+1)(n+2)}} \lambda^{\frac{n}{(n+1)(n+2)}} R^{\frac{\alpha}{(n+1)(n+2)}}\|f\|_2\,,
\end{equation}
where $p=\frac{2(n+1)}{n-1}$, and $\gamma,\lambda$ are as in the assumptions of Theorem \ref{thm-L2X}.

We further sort the unit cubes $B$ in $X'$ as follows:
\begin{enumerate}

\item Let $\beta$ be a dyadic number, and $\cb_\beta$ a sub-collection of the unit cubes in $X'$ such that for each $B$ in $\cb_\beta$, the lattice $K^2$-cube $\tilde B$ containing $B$ satisfies
$$
\|\eit f\|_{L^p(\tilde B)} \sim \beta\,.
$$
Denote the collection of relevant $K^2$-cubes by $\tilde\cb_\beta$.

\item Fix $\beta$. Let $\lambda'$ be a dyadic number and $\cb_{\beta,\lambda'}$ a sub-collection of $\cb_{\beta}$ such that for each $B\in \cb_{\beta,\lambda'}$, the lattice $R^{1/2}$-cube $Q$ containing $B$ contains $\sim \lambda'$ many $K^2$-cubes from $\tilde\cb_\beta$. Denote the collection of relevant $K^2$-cubes by $\tilde\cb_{\beta,\lambda'}$.

\end{enumerate}

Since there are only $O(\log R)$ many significant choices for all dyadic numbers $\beta,\lambda'$, we can choose some $\beta$ and $\lambda'$ so that $\#\cb_{\beta,\lambda'} \gtrapprox M$. Then it follows easily by definition that
$$
M':=\#\tilde \cb_{\beta,\lambda'} \gtrapprox M, \quad \lambda'\leq \lambda\,,
$$
and
$$
\gamma':= \max_{ \underset{x'\in \ZR^{n+1},r\geq K^2}{B^{n+1}(x',r)\subset B^{n+1}(0,R)}} \frac{\#\{\tilde B \in \tilde\cb_{\beta,\lambda'} :\tilde B \subset B(x',r)\}}{r^\alpha} \leq \gamma\,.$$

Applying Proposition \ref{thm-main} to $\|\eit f\|_{L^p(Y)}$ with $Y=\bigcup_{\tilde B \in \tilde\cb_{\beta,\lambda'}} \tilde B$ and parameters $M',\gamma',\lambda'$, we get
$$
\|\eit f\|_{L^p(X)} \lessapprox \|\eit f\|_{L^p(Y)} \lessapprox 
M^{-\frac{1}{n+1}}\gamma^{\frac{2}{(n+1)(n+2)}} \lambda^{\frac{n}{(n+1)(n+2)}} R^{\frac{\alpha}{(n+1)(n+2)}}\|f\|_2\,,
$$
as desired.

\end{proof}

The rest of this section is devoted to a proof of Proposition \ref{thm-main}. Note that when the radius $R$ is $\lesssim 1$, the estimate \eqref{eq-main} is trivial. So we can assume that $R$ is sufficiently large compared to any constant depending on $\e$. We will induct on radius $R$ in our proof.

In the proof, we will sometimes have paragraphs starting with \textbf{Intuition}. We hope that these will help the readers understand what we do next.

\begin{intu}
For our union $Y$ of $K^2$-cubes, we want to use decoupling theory on each $K^2$-cube. This will relate the whole $\eit f$ to its contributions $\eit f_{\tau}$ from various $1/K$-caps $\tau$ in the frequency space. Instead of doing decoupling in dimension $n+1$, we are going to do a broad-narrow analysis following Bourgain-Guth \cite{BG}, Bourgain \cite{jB12}, Bourgain-Demeter \cite{BD} and Guth \cite{G2}: for each $K^2$-cube, one of the following two has to happen:

(i) It is broad in the sense that there are $n+1$ contributing caps that are transversal. In this case the function is controlled by multilinear estimates which are usually strong enough. 

(ii) It is narrow (i.e. not broad). In this case all the contributing caps have normal directions close to a hyperplane, which enables us to use decoupling in dimension $n$.

Either way we get better estimates than a direct $(n+1)$-dimensional decoupling. We control the broad part directly, and do an induction on the narrow part. Our induction has its roots in the proof of the refined Strichartz estimate in \cite{DGL, DGLZ}.
\end{intu}

Throughout this section we fix $p=\frac{2(n+1)}{n-1}$.
In the frequency space we decompose $B^n(0,1)$ into disjoint $K^{-1}$-cubes $\tau$. Denote the set of $K^{-1}$-cubes $\tau$ by $\cs$.
For a function $f$ with ${\rm supp}\widehat{f}\subset B^n(0,1)$ we have $f=\sum_\tau f_\tau$, where $\wh{f_\tau}$ is $\wh f$ restricted to $\tau$. Given a $K^2$-cube $B$, we define its \textbf{significant} set as
$$
\cs(B):=\left\{\tau \in \cs : \|\eit f_\tau\|_{L^p(B)} \geq \frac{1}{100(\#\cs)}\|\eit f\|_{L^p(B)}\right\}\,.
$$
Note that due to the triangle inequality
$$
\big\|\sum_{\tau\in \cs(B)}\eit f_\tau\big\|_{L^p(B)} \sim \|\eit f\|_{L^p(B)}\,.
$$
We say that a $K^2$-cube $B$ is \textbf{narrow} if there is an $n$-dimensional subspace $V$ such that for all $\tau \in \cs(B)$
$$
{\rm Angle}(G(\tau),V) \leq \frac{1}{100nK}\,,
$$
where $G(\tau)\subset S^n$ is a spherical cap of radius $\sim K^{-1}$ given by
$$
G(\tau):=\left\{\frac{(-2\xi,1)}{|(-2\xi,1)|}\in S^n:\xi\in\tau\right\}\,,
$$
and ${\rm Angle}(G(\tau),V)$ denotes the smallest angle between any non-zero vector $v\in V$ and $v'\in G(\tau)$. Otherwise we say the $K^2$-cube $B$ is \textbf{broad}. It follows from this definition that for any broad $B$, there exist $\tau_1,\cdots \tau_{n+1} \in \cs(B)$ such that for any $v_j \in G(\tau_j)$
\begin{equation} \label{eq-trans}
|v_1 \wedge v_2\wedge \cdots \wedge v_{n+1}| \gtrsim K^{-n}\,.
\end{equation}
Denote the union of broad $K^2$-cubes $B_k$ in $Y$ by $Y_{\rm broad}$, and the union of narrow $K^2$-cubes $B_k$ in $Y$ by $Y_{\rm narrow}$. We call it the broad case if $Y_{\rm broad}$ contains $\geq M/2$ many $K^2$-cubes, and the narrow case otherwise. We will deal with the broad case in Subsection \ref{sec-br} using the multilinear refined Strichartz estimate from \cite{DGLZ}. And we handle the narrow case in Subsection \ref{sec-nr} by an inductive argument via the Bourgain-Demeter $l^2$ decoupling theorem \cite{BD} and induction on scales.

\subsection{Broad case} \label{sec-br}
Recall that $K=R^\delta$. A key tool we are using in the broad case is the following multilinear refined Strichartz estimate from \cite{DGLZ}, which is proved using $l^2$ decoupling, induction on scales and multilinear Kakeya estimates (see \cite{BCT,G}).

\begin{theorem} [c.f. Theorem 4.2 in \cite{DGLZ}] \label{multstr}
Let $q=\frac{2(n+2)}{n}$. Let $f$ be a function with Fourier support in $B^n(0,1)$. Suppose that $\tau_1,\cdots,\tau_{n+1} \in \cs$ and \eqref{eq-trans} holds for any $v_j\in G(\tau_j)$. Suppose that $Q_1, Q_2,\cdots, Q_N$ are lattice  $R^{1/2}$-cubes in $B_R^{n+1}$, so that
$$
\| \eit f_{\tau_i} \|_{L^{q}(Q_j)}  \textrm{ is dyadically a constant in $j$, for each $i=1,2,\cdots,n+1$}. 
$$
Let $Y$ denote $\bigcup_{j=1}^N Q_j$.  Then
for any $\eps > 0$, 
\begin{equation} \label{kRS}
\left\| \prod_{i=1}^{n+1} \left|\eit f_{\tau_i}\right|^{\frac{1}{n+1}} \right\|_{L^{q}(Y)} \le C_\e R^\e N^{-\frac{n}{(n+1)(n+2)}} \|f\|_2\,. 
\end{equation}
\end{theorem}

Throughout the remainder of this subsection we will prove Proposition \ref{thm-main} in the broad case. In the broad case, there are $\sim M$ many broad $K^2$-cubes $B$. Denote the collection of $(n+1)$-tuple of transverse caps by $\Gamma$:
$$
\Gamma:=\left\{\tilde\tau=(\tau_1,\cdots,\tau_{n+1}):\tau_j\in \cs\text{ and } \eqref{eq-trans} \text{ holds for any } v_j\in G(\tau_j)\right\}\,.
$$
Then for each broad $B$, 
\begin{equation} \label{eq-br}
\left\|\eit f\right\|^p_{L^p(B)} \leq K^{O(1)} \prod_{j=1}^{n+1}\left(\int_B \big|\eit f_{\tau_j}\big|^p\right)^{\frac{1}{n+1}}\,,
\end{equation}
for some $\tilde\tau=(\tau_1,\cdots,\tau_{n+1})\in \Gamma$. In order to exploit the transversality, we want to bound the above geometric average of integrals by an integral of geometric average up to a loss of $K^{O(1)}$. We can do this by using translations and locally constant property. Given a $K^2$-cube $B$, denote its center by $x_B$.
We break $B$ into finitely overlapping balls of the form $B(x_B+v,2)$, where $v\in B(0,K^2)\cap \ZZ^{n+1}$. For each $\tau_j$, we can view $|\eit f_{\tau_j}|$ essentially as constant on each $B(x_B+v,2)$. Choose $v_j\in B(0,K^2)\cap \ZZ^{n+1}$ such that $\|\eit f_{\tau_j}\|_{L^\infty(B)}$ is attained in $B(x_B+v_j,2)$. Denote $v_j=(x_j,t_j)$ and define $f_{\tau_j,v_j}$ by
$$
\widehat{f_{\tau_j,v_j}} (\xi) := \widehat{f_{\tau_j}} (\xi) e^{i(x_j\cdot \xi +t_j|\xi|^2)}\,.
$$
Then $\eit f_{\tau_j,v_j} (x) = e^{i(t+t_j)\Delta} f_{\tau_j} (x+x_j)$ and $|\eit f_{\tau_j,v_j} (x)|$ attains $\|\eit f_{\tau_j}\|_{L^\infty(B)}$ in $B(x_B,2)$. Therefore
\begin{equation} \label{eq-rt}
\int_B \big|\eit f_{\tau_j}\big|^p \leq K^{O(1)} \int_{B(x_B,2)} \big|\eit f_{\tau_j,v_j}\big|^p\,.
\end{equation}
Now for each broad $B$, we find some $\tilde\tau=(\tau_1,\cdots,\tau_{n+1})\in \Gamma$ and $\tilde v =(v_1,\cdots,v_{n+1})$ such that 
\begin{equation} \label{eq-lcst}
\begin{split}
\left\|\eit f\right\|^p_{L^p(B)} \leq &K^{O(1)} \prod_{j=1}^{n+1}\left(\int_{B(x_B,2)} \big|\eit f_{\tau_j,v_j}\big|^p\right)^{\frac{1}{n+1}}\\
\leq & K^{O(1)} \int_{B(x_B,2)} \prod_{j=1}^{n+1}\left|\eit f_{\tau_j,v_j}\right|^{\frac{p}{n+1}}\,.
\end{split}
\end{equation}
Since there are only $K^{O(1)}$ choices for $\tilde \tau$ and $\tilde v$, we can choose some $\tilde \tau$ and $\tilde v$ such that \eqref{eq-lcst} holds for at least $K^{-C} M$ broad balls $B$. From now on, fix $\tilde \tau$ and $\tilde v$, and let $f_j$ denote $f_{\tau_j,v_j}$. Next we further sort the collection $\cb$ of remaining broad balls as follows:
\begin{enumerate}

\item For a dyadic number $A$, let $\cb_A$ be a sub-collection of $\cb$ in which for each $B$ we have 
$$\left\|\prod_{j=1}^{n+1}\left|\eit f_j\right|^{\frac{1}{n+1}}\right\|_{L^\infty(B(x_B,2))} \sim A\,.$$

\item Fix $A$, for dyadic numbers $\tilde \lambda,\iota_1,\cdots,\iota_{n+1}$, let $\cb_{A,\tilde \lambda,\iota_1,\cdots,\iota_{n+1}}$ be a sub-collection of $\cb_A$ in which for each $B$, the $R^{1/2}$-cube $Q$ containing $B$ contains $\sim \tilde \lambda$ cubes from $\cb_A$ and 
$$
\|\eit f_j\|_{L^q(Q)} \sim \iota_j, \quad j=1,2,\cdots,n+1\,.
$$
Here $q=\frac{2(n+2)}{n}$.
\end{enumerate}
Recall that $p>q$, where $p=\frac{2(n+1)}{n-1}$ is the sharp exponent for decoupling in dimension $n$, and $q=\frac{2(n+2)}{n}$ is the exponent for which the multilinear refined Strichartz estimate in demension $n+1$ holds.
The first dyadic pigeonholing together with the locally constant property enables us to dominate $L^p$-norm by $L^q$-norm using reverse H\"older. The second dyadic pigeonholing allows us to apply the multilinear refined Strichartz estimate to control the $L^q$-norm.

We can assume that $\|f\|_2=1$. Then all the above dyadic numbers making significant contributions can be assumed to be between $R^{-C}$ and $R^{C}$ for a large constant $C$. Therefore, there exist some dyadic numbers $A,\tilde \lambda,\iota_1,\cdots,\iota_{n+1}$ such that $\cb_{A,\tilde \lambda,\iota_1,\cdots,\iota_{n+1}}$ contains $\geq K^{-C} M$ many cubes $B$. Fix a choice of $A,\tilde \lambda,\iota_1,\cdots,\iota_{n+1}$ and denote $\cb_{A,\tilde \lambda,\iota_1,\cdots,\iota_{n+1}}$ by $\cb$ for convenience (a mild abuse of notation). Then, in the broad case, it follows from \eqref{eq-lcst} and our choice of $A$ that
\begin{equation} \label{all'}
\begin{split}
\|\eit f\|_{L^p(Y)} \leq &K^{O(1)} \left\| \prod_{j=1}^{n+1}\left|\eit f_j\right|^{\frac{1}{n+1}} \right\|_{L^p(\cup_{B\in \cb}B(x_B,2))}\\
\leq &K^{O(1)} M^{\frac 1p -\frac 1 q} \left\| \prod_{j=1}^{n+1}\left|\eit f_j\right|^{\frac{1}{n+1}} \right\|_{L^q(\cup_{B\in \cb}B(x_B,2))}\\
\leq &K^{O(1)} M^{-\frac {1}{(n+1)(n+2)}} \left\| \prod_{j=1}^{n+1}\left|\eit f_j\right|^{\frac{1}{n+1}} \right\|_{L^q(\cup_{Q\in \cq} Q)}\,,
\end{split}
\end{equation}
where $\cq$ is the collection of relevant $R^{1/2}$-cubes $Q$ when we define $\cb$. Note that 
$$
(\#\cq) \lambda  \geq (\#\cq) \tilde \lambda \sim \#\cb \geq K^{-C}M\,,
$$
so
\begin{equation}
\tilde N :=\#\cq \geq K^{-C}\frac{M}{\lambda}\,.
\end{equation}
Applying Theorem \ref{multstr}, we get
$$
\left\| \prod_{j=1}^{n+1}\left|\eit f_j\right|^{\frac{1}{n+1}} \right\|_{L^q(\cup_{Q\in \cq} Q)}
\leq K^{O(1)} \left(\frac M\lambda\right)^{-\frac{n}{(n+1)(n+2)}}
\|f\|_2\,,$$
and therefore by \eqref{all'},
$$
\|\eit f\|_{L^p(Y)} \leq K^{O(1)} M^{-\frac{1}{n+2}} \lambda^{\frac{n}{(n+1)(n+2)}}
\|f\|_2\,.$$
Note that $$M^{-\frac{1}{n+2}} \lambda^{\frac{n}{(n+1)(n+2)}} \leq K^{O(1)} M^{-\frac{1}{n+1}}\gamma^{\frac{2}{(n+1)(n+2)}} \lambda^{\frac{n}{(n+1)(n+2)}} R^{\frac{\alpha}{(n+1)(n+2)}}$$
holds if and only if $M\leq K^{O(1)}\gamma^2R^\alpha$. Indeed, by definition \eqref{ga} of $\gamma$, we have $M\leq \gamma R^\alpha$ and $\gamma \geq K^{-2\alpha}$. So the broad case is done.

\subsection{Narrow case} \label{sec-nr}

For each narrow ball, we have the following lemma which is a consequence of the $l^2$ decoupling theorem in dimension $n$ and Minkowski's inequality. This argument is essentially contained in Bourgain-Demeter's proof of the $l^2$ decoupling conjecture and we omit the details (see the proof of Proposition 5.5 in \cite{BD}). 

\begin{lemma}\label{lem-nr-dec}
Suppose that $B$ is a narrow $K^2$-cube in $\ZR^{n+1}$. Then for any $\e>0$,
$$
\|\eit f\|_{L^p(B)} \leq C_\e K^{\e} \left(\sum_{\tau\in \cs}\left\|\eit f_\tau\right\|_{L^p(\omega_B)}^2\right)^{1/2} \,,
$$
here $p=\frac{2(n+1)}{n-1}$, $\cs$ denotes the set of $K^{-1}$-cubes which tile $B^n(0,1)$, and $\omega_B$ is a weight function which is essentially a characteristic function on $B$. More precisely, $\omega_B$ has Fourier support in $B(0,K^{-2})$ and satisfies
$$
1_B(\tilde x) \lesssim \omega_B(\tilde x) \leq \left(1+\frac{|\tilde x - C(B)|}{K^2}\right)^{-1000n}.
$$
\end{lemma}

For each $\tau \in \cs$, we will deal with $\eit f_\tau$ by parabolic rescaling and induction on radius. In order to do so, we need to further decompose $f$ in physical space and perform dyadic pigeonholing several times to get the right setup for our inductive hypothesis at scale $R_1:=R/K^2$ after rescaling. 

\begin{intu}
For each $1/K$-cap $\tau$, all wave packets associated with $f_\tau$ through a given point have to lie in a common box that has one side length $R$ and other side lengths $R/K$. Every single box of this type will become an $R/K^2$-ball if we perform a parabolic rescaling to transform $\tau$ into the standard $1$-cap. We want to use the inductive hypothesis for radius $R/K^2$ in an efficient way. A few dyadic pigeonholing steps will be needed.
\end{intu}

First, we break the physical ball $B^n(0,R)$ into $R/K$-cubes $D$. For each pair $(\tau, D)$, let $f_{\Box_{\tau, D}}$ be the function formed by cutting off $f$ on the cube $D$ (with a Schwartz tail) in physical space and the cube $\tau$ in Fourier space. Note that $e^{it\Delta} f_{\Box_{\tau,D}}$, restricted to $B^{n+1}(R)$, is essentially supported on an $R/K \times \cdots \times R/K \times R$-box\footnote[3]{In reality, our boxes will have edge length slightly larger, say being larger by $K^{\varepsilon^{100}}$ times. See e.g. the wave packet decomposition theorem in \cite{G1}. This would not hurt us in any way and we omit this technicality for reading convenience.}, which we denote by $\Box_{\tau, D}$. The box $\Box_{\tau, D}$ is in the direction given by $(-2c(\tau),1)$ and intersects ${t=0}$ at the cube $D$, where $c(\tau)$ is the center of $\tau$. For a fixed $\tau$, the different boxes $\Box_{\tau, D}$ tile $B^{n+1}(0,R)$.  In particular, for each $\tau$, a given $K^2$-cube $B$ lies in exactly one box $\Box_{\tau, D}$. We write $f=\sum_{\Box} f_\Box$ for abbreviation. By Lemma \ref{lem-nr-dec}, for each narrow $K^2$-cube $B$, 
\begin{equation} \label{eq-dec}
\|\eit f\|_{L^p(B)} \lesssim  K^{\e^4} \left(\sum_{\Box}\left\|\eit f_\Box\right\|_{L^p(\omega_B)}^2\right)^{1/2} \,.
\end{equation}

\begin{figure}  [ht]
\centering
\includegraphics[scale=.3]{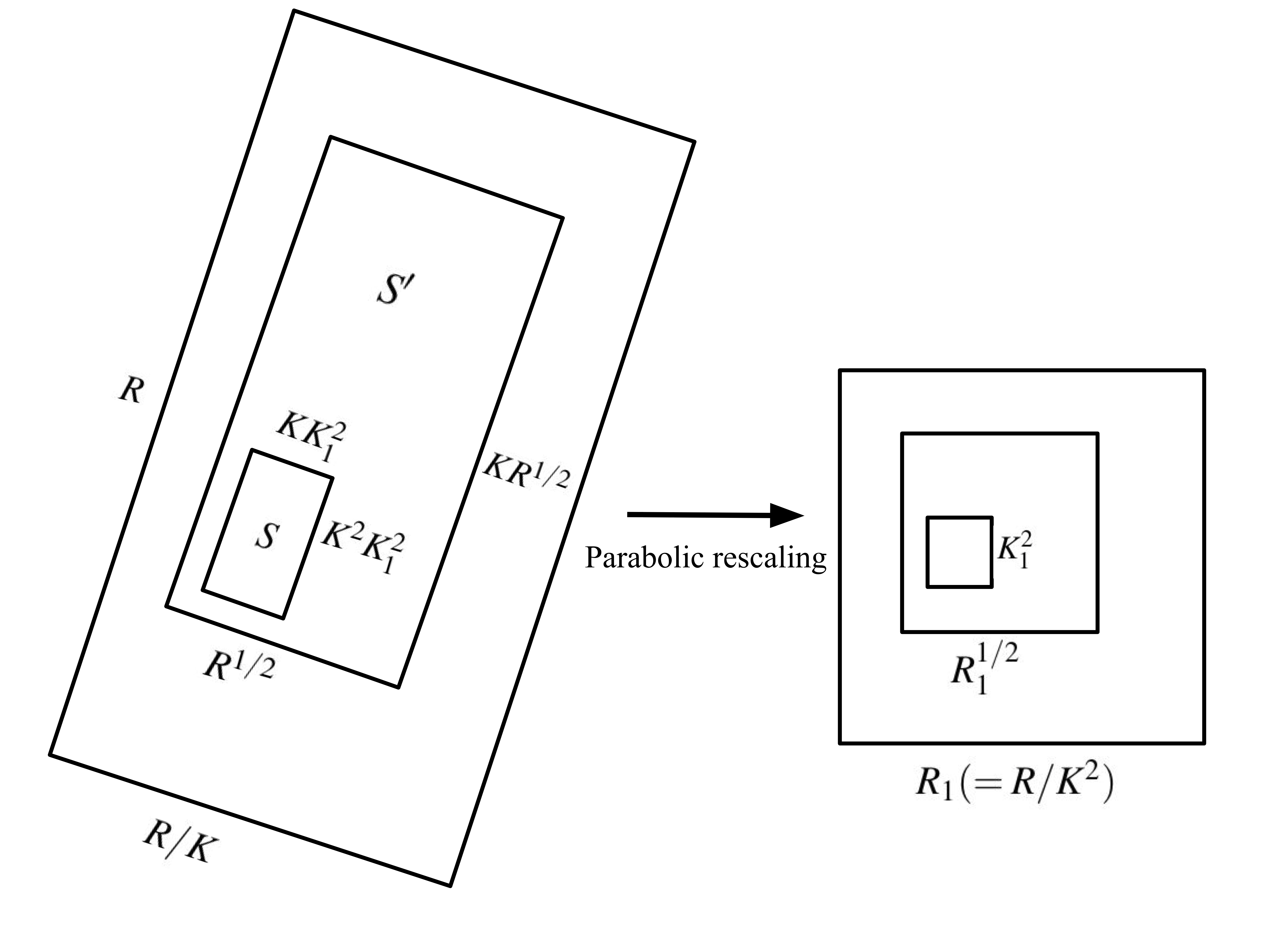}
\caption{\small{ Tubes of different scales in the $\Box$}}
\label{fig-stri}
\end{figure}

 We will have a gain $\frac{1}{K^{2\e}}$ from induction on radius. Therefore, in \eqref{eq-dec} we are allowed to lose a small power of $K$. This small power depends on $\e$ and should be smaller than $2\e$. It could be $\e^2,\e^3,\e^4,$ etc.

Next, we perform a dyadic pigeonholing to get our inductive hypothesis for each $f_{\Box}$. Recall that $K=R^{\delta}$, where $\delta=\e^{100}$. Denote
$$
R_1:=R/K^2=R^{1-2\delta}, \quad K_1 := R_1^\delta=R^{\delta-2\delta^2}\,.
$$
Tile $\Box$ by $KK_1^2\times \cdots \times KK_1^2 \times K^2K_1^2$-tubes $S$, and also tile $\Box$ by $R^{1/2}\times \cdots \times R^{1/2} \times KR^{1/2}$-tubes $S'$ (all running parallel to the long axis of $\Box$). 
To understand these scales, see Figure \ref{fig-stri} for the change in physical space \eqref{coord} during the process of parabolic rescaling. In particular, after rescaling the $\Box$ becomes an $R_1$-cube, the tubes $S'$ and $S$ become lattice $R_1^{1/2}$-cubes and $K_1^2$-cubes respectively.

We apply the following to regroup tubes $S$ and $S'$ inside each $\Box$:

\begin{enumerate}

\item Sort those tubes $S$ which intersect $Y$ according to the value $\|\eit f_\Box\|_{L^p(S)}$ and the number of narrow $K^2$-cubes contained in it. For dyadic numbers $\eta, \beta_1$, we use $\ZS_{\Box,\eta,\beta_1}$ to stand for the collection of tubes $S\subset \Box$ each of which containing $\sim \eta$ narrow $K^2$-cubes in $Y_{\rm narrow}$ and $\|\eit f_\Box\|_{L^p(S)} \sim \beta_1$. 

\item For fixed $\eta,\beta_1$, we sort the tubes $S'\subset \Box$ according to the number of tubes $S\in \ZS_{\Box,\eta,\beta_1}$ contained in it. For dyadic number $\lambda_1$, let $\ZS_{\Box,\eta,\beta_1,\lambda_1}$ be the sub-collection of $\ZS_{\Box,\eta,\beta_1}$ such that for each $S\in \ZS_{\Box, \eta,\beta_1,\lambda_1}$, the tube $S'$ containing $S$ contains $\sim \lambda_1$ tubes from $\ZS_{\Box,\eta, \beta_1}$.

\item For fixed $\eta, \beta_1,\lambda_1$, we sort the boxes $\Box$ according to the value $\|f_\Box\|_2$, the number $\#\ZS_{\Box, \eta, \beta_1, \lambda_1}$ and the value $\gamma_1$ defined below. For dyadic numbers $\beta_2, M_1,\gamma_1$, let $\ZB_{\eta,\beta_1,\lambda_1,\beta_2, M_1,\gamma_1}$ denote the collection of boxes $\Box$ each of which satisfying that $$\|f_\Box\|_2\sim \beta_2, \quad \#\ZS_{\Box,\eta,\beta_1,\lambda_1} \sim M_1$$
and
\begin{equation} \label{ga1}
\max_{T_{r}\subset \Box:r\geq K_1^2} \frac{\#\{S\in \ZS_{\Box,\eta,\beta_1,\lambda_1}: S\subset T_r\}}{r^\alpha} \sim \gamma_1\,,
\end{equation}
where $T_r$ are $Kr \times \cdots \times Kr \times K^2 r$-tubes in $\Box$ running parallel to the long axis of $\Box$.

\end{enumerate}

Let $Y_{\Box,\eta,\beta_1,\lambda_1}$ be the union of the tubes $S$ in $\ZS_{\Box,\eta,\beta_1,\lambda_1}$, and $\hichi_{Y_{\Box,\eta,\beta_1,\lambda_1}}$ the corresponding characteristic function. Then on $Y_{\rm narrow}$ we can write
$$
\eit f = \sum_{\eta,\beta_1,\lambda_1,\beta_2, M_1,\gamma_1}\left( \sum_{\Box \in \ZB_{\eta,\beta_1,\lambda_1,\beta_2, M_1,\gamma_1}} \eit f_{\Box} \cdot \hichi_{Y_{\Box,\eta,\beta_1,\lambda_1}} \right) + O(R^{-1000n}) \|f\|_2\,.
$$

The small error term $O(R^{-1000n}) \|f\|_2$ will prove to be harmless in our computations. We will neglect this term in the sequel.
Again, to make the statement really rigorous one needs to increase the side lengths of $\Box$ by a tiny power of $R$, say $R^{{\delta}^{100}} \sim K^{\delta^{99}}$. As before, we choose to ignore this technicality
in order to facilitate the main exposition.

In particular, on each narrow $B$ we have
\begin{equation} \label{fB}
\eit f = \sum_{\eta,\beta_1,\lambda_1,\beta_2, M_1,\gamma_1}\left( \sum_{\underset{B\subset Y_{\Box,\eta,\beta_1,\lambda_1}}{\Box \in \ZB_{\eta,\beta_1,\lambda_1,\beta_2, M_1,\gamma_1}}} \eit f_{\Box} \right)\,.
\end{equation}

Without loss of generality, we assume that $\|f\|_2=1$. Then we can further assume that the dyadic numbers above are in reasonable ranges, say
$$
1\leq \eta \leq K^{O(1)}, \quad R^{-C}\leq \beta_1 \leq K^{O(1)}, \quad 1\leq \lambda_1 \leq R^{O(1)}
$$
and
$$
R^{-C} \leq \beta_2 \leq 1, \quad 1\leq M_1\leq R^{O(1)},\quad K^{-2n}\leq \gamma_1\leq R^{O(1)}\,,
$$
where $C$ is a large constant such that the contributions from those $\beta_1$ and $\beta_2$ less than $R^{-C}$ are negligible. Therefore, there are only $O(\log R)$ significant choices for each dyadic number. Because of \eqref{eq-dec} and \eqref{fB}, by pigeonholing, we can choose $\eta,\beta_1,\lambda_1,\beta_2, M_1,\gamma_1$ so that 
\begin{equation} \label{dya1}
\|\eit f\|_{L^p(B)} \lesssim  (\log R)^{6} K^{\e^4} \left(\sum_{\underset{B\subset Y_{\Box,\eta,\beta_1,\lambda_1}}{\Box \in \ZB_{\eta,\beta_1,\lambda_1,\beta_2, M_1,\gamma_1}}}\|\eit f_\Box\|_{L^p(\omega_B)}^2\right)^{1/2}
\end{equation}
holds for a fraction $\gtrsim (\log R)^{-6}$ of all narrow $K^2$-cubes $B$.

We fix $\eta,\beta_1,\lambda_1,\beta_2, M_1,\gamma_1$ for the rest of the proof. Let $Y_\Box$ and $\ZB$ stand for the abbreviations of $Y_{\Box,\eta,\beta_1,\lambda_1}$ and $\ZB_{\eta,\beta_1,\lambda_1,\beta_2, M_1,\gamma_1}$ respectively. Finally we sort the narrow balls $B$ satisfying \eqref{dya1} by $
\#\{\Box\in \ZB: B\subset Y_{\Box}\}$. Let $Y'\subset Y_{\rm narrow}$ be a union of narrow $K^2$-cubes $B$ each of which obeying
\begin{equation} \label{dya1'}
\|\eit f\|_{L^p(B)} \lesssim (\log R)^6 K^{\e^4} \left(\sum_{\Box\in \ZB: B\subset Y_\Box}\|\eit f_\Box\|_{L^p(\omega_B)}^2\right)^{1/2}
\end{equation}
and
\begin{equation} \label{dya2}
\#\{\Box\in \ZB: B\subset Y_{\Box}\} \sim \mu
\end{equation}
for some dyadic number $1\leq \mu \leq K^{O(1)}$, moreover the number of $K^2$-cubes $B$ in $Y'$ is $\gtrsim$ $(\log R)^{-7} M$. 

Now we are done with dyadic pigeonholing argument and let us put all these together. By our assumption that $\|\eit f\|_{L^p(B_k)}$ is essentially constant in $k=1,2,\cdots,M$, in the narrow case we have
\begin{equation} \label{eq-Y'}
\|\eit f\|_{L^p(Y)}^p \lesssim (\log R)^7 \sum_{B\subset Y'} \|\eit f\|^p_{L^p(B)}\,.
\end{equation}
For each $B\subset Y'$, it follows from \eqref{dya1'}, \eqref{dya2} and H\"older's inequality that
\begin{equation} \label{eq-H}
\|\eit f\|^p_{L^p(B)} \lesssim  (\log R)^{6p} K^{\e^4p} \mu^{\frac{p}{2}-1} \sum_{\Box\in \ZB: B\subset Y_{\Box}}\|\eit f_\Box\|_{L^p(\omega_B)}^p\,.
\end{equation}
Putting \eqref{eq-Y'} and \eqref{eq-H} together and as before omiting the rapidly decaying tails,
\begin{equation} \label{Y-Ybox}
\|\eit f\|_{L^p(Y)} \lesssim (\log R)^{13} K^{\e^4} \mu^{\frac{1}{n+1}} \left(\sum_{\Box\in \ZB} \left\|\eit f_\Box\right\|^p_{L^p(Y_\Box)}\right)^{1/p}\,.
\end{equation}

Next, to each $\|\eit f_\Box\|_{L^p(Y_\Box)}$ we apply parabolic rescaling and induction on radius. For each $1/K$-cube $\tau=\tau_{\Box}$ in $B^n(0,1)$, we write $\xi=\xi_0+K^{-1} \zeta \in \tau$, where $\xi_0$ is the center of $\tau$. Then 
$$
|\eit f_{\Box} (x)| =K^{-n/2} |e^{i\tilde t \Delta} g (\tilde x)|
$$
for some function $g$ with Fourier support in the unit cube and $\|g\|_2=\|f_{\Box}\|_2$, where the new coordinates $(\tilde x,\tilde t)$ are related to the old coordinates $(x,t)$ by
\begin{equation} \label{coord}
\begin{cases}
   \tilde x =K^{-1} x + 2 t K^{-1} \xi_0\,, \\
   \tilde t = K^{-2} t \,.
\end{cases}
\end{equation}
For simplicity, denote the above relation by $(\tilde x,\tilde t)=F(x,t)$.
Therefore
\begin{equation} \label{Ybox-Ytilde}
\|\eit f_\Box (x)\|_{L^p(Y_\Box)} = K^{\frac{n+2}{p}-\frac{n}{2}} \|e^{i\tilde t \Delta} g(\tilde x)\|_{L^p(\tilde Y)}=
K^{-\frac{1}{n+1}} \|e^{i\tilde t \Delta} g(\tilde x)\|_{L^p(\tilde Y)},
\end{equation}
where $\tilde Y$ is the image of $Y_\Box$ under the new coordinates. 

Note that we can apply our inductive hypothesis \eqref{eq-main} at scale $R_1=R/K^2$ to $\|e^{i\tilde t \Delta} g(\tilde x)\|_{L^p(\tilde Y)}$ with new parameters $M_1, \gamma_1, \lambda_1, R_1$. More precisely, $\tilde Y = F(Y_{\Box})$ consists of $\sim M_1$ distinct $K_1^2$-cubes $F(S)$ in an $R_1$-ball $F(\Box)$, and the $K_1^2$-cubes $F(S)$ are organized into $R_1^{1/2}$-cubes $F(S')$ such that each cube $F(S')$ contains $\sim \lambda_1$ cubes $F(S)$. Moreover, $\|e^{i\tilde t \Delta} g(\tilde x)\|_{L^p(F(S))}$ is dyadically a constant in $S \subset Y_{\Box}$. By our choice of $\gamma_1$, we have
$$
\max_{ \underset{x'\in \ZR^{n+1},r\geq K_1^2}{B^{n+1}(x',r)\subset F(\Box)} }\frac{\#\{F(S) : F(S) \subset B(x',r)\}}{r^\alpha} \sim \gamma_1\,.
$$
Henceforth, by \eqref{Ybox-Ytilde} and inductive hypothesis \eqref{eq-main} at scale $R_1$ we have
\begin{equation} \label{ind}
\begin{split}
&\|\eit f_\Box (x)\|_{L^p(Y_\Box)} \\ \lesssim &K^{-\frac{1}{n+1}} M_1^{-\frac{1}{n+1}}\gamma_1^{\frac{2}{(n+1)(n+2)}} \lambda_1^{\frac{n}{(n+1)(n+2)}} \left(\frac{R}{K^2}\right)^{\frac{\alpha}{(n+1)(n+2)}+\e}\|f_\Box\|_2\,.
\end{split}
\end{equation}
From \eqref{Y-Ybox} and \eqref{ind} we obtain
\begin{equation} \label{all}
\begin{split}
&\|\eit f\|_{L^p(Y)}\\ \lesssim &K^{2\e^4} \mu^{\frac{1}{n+1}} 
K^{-\frac{1}{n+1}} M_1^{-\frac{1}{n+1}}\gamma_1^{\frac{2}{(n+1)(n+2)}} \lambda_1^{\frac{n}{(n+1)(n+2)}} \left(\frac{R}{K^2}\right)^{\frac{\alpha}{(n+1)(n+2)}+\e} \left(\sum_{\Box\in \ZB}\|f_\Box\|_2^p\right)^{1/p} \\
\lesssim & K^{2\e^4} \left(\frac{\mu}{\#\ZB} \right)^{\frac{1}{n+1}} 
K^{-\frac{1}{n+1}} M_1^{-\frac{1}{n+1}}\gamma_1^{\frac{2}{(n+1)(n+2)}} \lambda_1^{\frac{n}{(n+1)(n+2)}} \left(\frac{R}{K^2}\right)^{\frac{\alpha}{(n+1)(n+2)}+\e} \|f\|_2 \,,
\end{split}
\end{equation}
where the last inequality follows from orthogonality $\sum_\Box\|f_\Box\|_2^2 \lesssim \|f\|_2^2$ and the assumption that $\|f_\Box\|_2 \sim$ constant in $\Box\in \ZB$.

\begin{intu}To finish our inductive argument, we have to relate the old and new parameters. Our setup allows us to do this in a nice way: Given $M_1, \lambda_1$ and $\gamma_1$, if $\eta$ is small, i.e. each $S$ contains very few narrow $K^2$-cubes, then $M$ is relatively small; if $\eta$ is large, i.e. each $S$ contains a lot of narrow $K^2$-cubes, then $\lambda$ and $\gamma$ are relatively large. Both make the right-hand side of what we want to prove reasonably large. This is the reason why one could believe the numerology will work out.
\end{intu}

Consider the cardinality of the set 
$\{(\Box, B): \Box \in \ZB, B\subset Y_\Box\cap Y'\}$.
By our choice of $\mu$ as in \eqref{dya2}, there is a lower bound 
$$\#\{(\Box, B): \Box \in \ZB, B\subset Y_\Box\cap Y'\} \gtrsim (\log R)^{-7} M \mu\,.$$ 
On the other hand, by our choices of $M_1$ and $\eta$, for each $\Box\in \ZB$, $Y_\Box$ contains $\sim M_1$ tubes $S$ and each $S$ contains $\sim \eta$ narrow cubes in $Y_{\rm narrow}$, so
$$
\#\{(\Box, B): \Box \in \ZB, B\subset Y_\Box\cap Y'\} \lesssim (\#\ZB)M_1\eta\,.
$$
Therefore, we get
\begin{equation} \label{muB}
\frac{\mu}{\#\ZB} \lesssim \frac{(\log R)^7 M_1\eta}{M}\,.
\end{equation}

Next by our choices of $\gamma_1$ as in \eqref{ga1} and $\eta$, 
\begin{equation*}
\begin{split}
&\gamma_1 \cdot \eta\\
\sim &
\max_{T_{r}\subset \Box:r\geq K_1^2} \frac{\#\{S: S \subset Y_\Box\cap T_r\}}{r^\alpha} \cdot \#\{B: B \subset S\cap Y_{\rm narrow} \text{ for any fixed } S\subset Y_\Box\} \\
\lesssim & \max_{T_{r}\subset \Box:r\geq K_1^2} \frac{\#\{B\subset Y:B \subset T_r\}}{r^\alpha} \leq \frac{K\gamma (Kr)^\alpha}{r^\alpha}= \gamma K^{\alpha+1}\,,
\end{split}
\end{equation*}
where the last inequality follows from the definition \eqref{ga} of $\gamma$ and the fact that we can cover a $Kr \times \cdots \times Kr \times K^2 r$-tube $T_r$ by $\sim K$ finitely overlapping $Kr$-balls. Hence,
\begin{equation} \label{eta}
\eta \lesssim \frac{\gamma K^{\alpha+1}}{\gamma_1}\,.
\end{equation}

Finally we relate $\lambda_1$ and $\lambda$ by considering the number of narrow $K^2$-balls in each relevant $R^{1/2}\times \cdots \times R^{1/2} \times KR^{1/2}$-tube $S'$. Recall that each relevant $S'$ contains $\sim \lambda_1$ tubes $S$ in $Y_\Box$ and each such $S$ contains $\sim \eta$ narrow balls. On the other hand, we can cover $S'$ by $\sim K$ finitely overlapping $R^{1/2}$-balls and by assumption each $R^{1/2}$-ball contains $\lesssim \lambda$ many $K^2$-cubes in $Y$. Thus it follows that
\begin{equation}\label{la-la1}
\lambda_1\lesssim \frac{K\lambda}{\eta}\,.
\end{equation}

By inserting \eqref{muB} and \eqref{la-la1} into \eqref{all},
\begin{equation*}
\begin{split}
&\|\eit f\|_{L^p(Y)} \\
\lesssim &\frac{K^{3\e^4}}{K^{2\e}}\left(\frac{\eta\gamma_1}{K^{\alpha+1}}\right)^{\frac{2}{(n+1)(n+2)}} M^{-\frac{1}{n+1}} \lambda^{\frac{n}{(n+1)(n+2)}} R^{\frac{\alpha}{(n+1)(n+2)}+\e}\|f\|_2 \\
\lesssim & \frac{K^{3\e^4}}{K^{2\e}} M^{-\frac{1}{n+1}}\gamma^{\frac{2}{(n+1)(n+2)}} \lambda^{\frac{n}{(n+1)(n+2)}} R^{\frac{\alpha}{(n+1)(n+2)}+\e}\|f\|_2\,,
\end{split}
\end{equation*}
where the last inequality follows from \eqref{eta}. Since $K=R^{\delta}$ and $R$ can be assumed to be sufficiently large compared to any constant depending on $\e$, we have $\frac{K^{3\e^4}}{K^{2\e}} \ll 1$ and the induction closes for the narrow case. This completes the proof of Proposition \ref{thm-main}.

\subsection{Remark} \label{sec-rmk} 
In Section \ref{sec-app}, we have seen that Corollary \ref{cor-L2X} is a direct result of Theorem \ref{thm-L2X}, and they are equally useful in applications to the sharp $L^2$ estimate of Schr\"odinger maximal function. We can also prove Corollary \ref{cor-L2X} from scratch using a similar argument as in this section, which is slightly easier in two aspects compared to that of Theorem \ref{thm-L2X}. First, in the broad case, it is sufficient to use multilinear restriction estimates and not necessary to invoke the multilinear refined Strichartz. Secondly, because there is one parameter less, the dyadic pigeonholing argument in the narrow case would be slightly reduced, for example, see Figure \ref{fig-restr} for tubes of different scales in the $\Box$ under the setting of Corollary \ref{cor-L2X}. 

\begin{figure}[H]
\centering
\includegraphics[scale=.27]{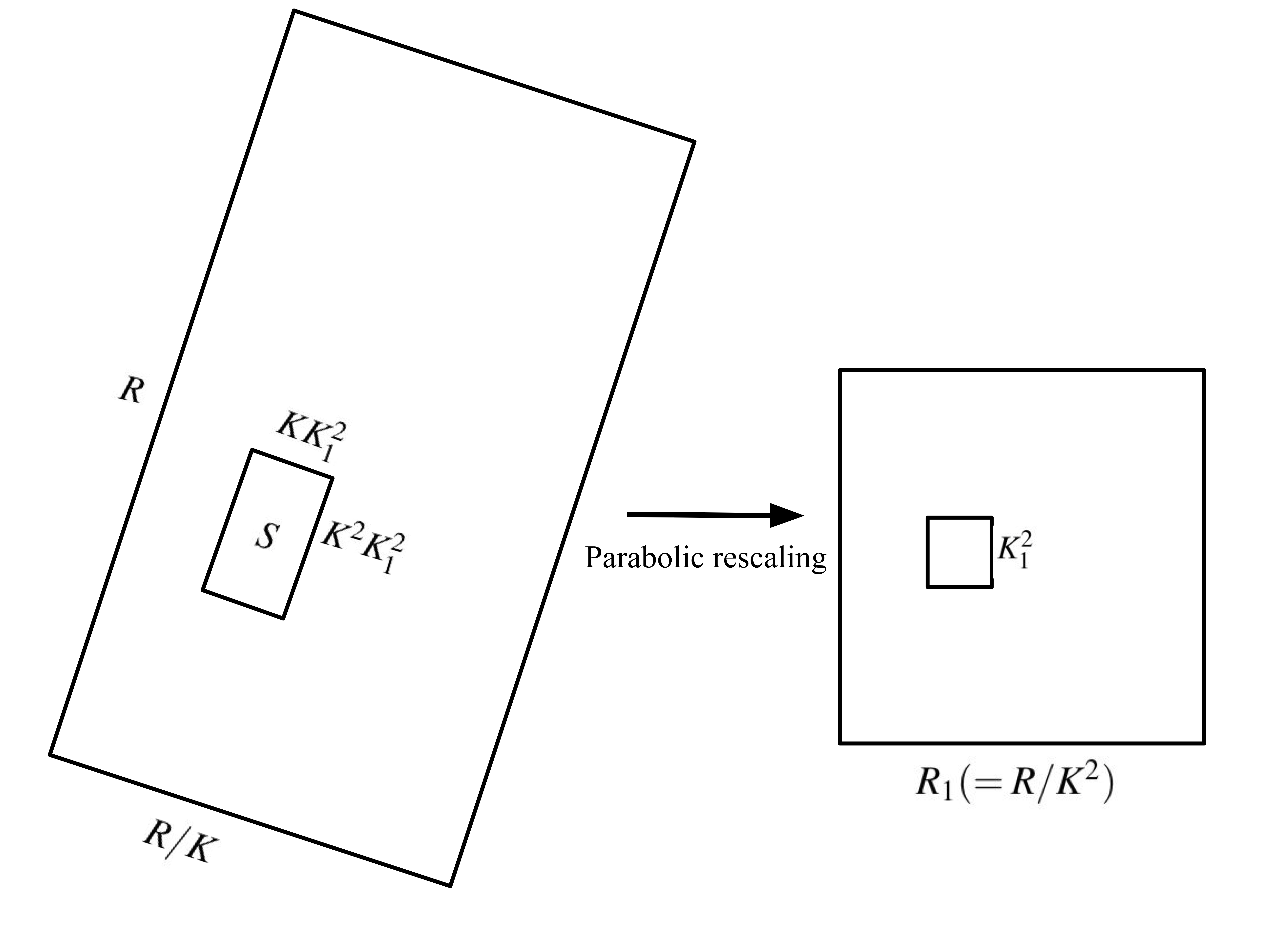}
\caption{\small{Tubes of different scales in the $\Box$} (in inductive argument for Corollary \ref{cor-L2X})}
\label{fig-restr}
\end{figure}

In fact, an adaptation of some arguments in the work \cite{W} of Wolff on the Falconer distance set problem in dimension $2$ can already imply Corollary \ref{cor-L2X} when $n=1$. In the special case $n=1$, the broad versus narrow dichotomy becomes the one on bilinear versus linear. To handle the linear part, the idea of induction on scales and splitting the ball into rectangular boxes ``$\Box$'' of size $R \times R/K$ in our proof already existed in Wolff's paper. We thank Hong Wang for pointing this out to us.

\end{document}